\documentclass[11pt]{article}

\usepackage[utf8]{inputenc}
\usepackage{algorithm}
\usepackage{algorithmic}
\usepackage{amsfonts}
\usepackage{amsmath,amssymb,amsfonts}
\usepackage{amsthm} 
\usepackage{booktabs}
\usepackage{xcolor}
\usepackage{enumerate}
\usepackage{graphicx}
\usepackage{latexsym}
\usepackage{listings}
\usepackage{longtable}
\usepackage{manyfoot}
\usepackage{mathtools}
\usepackage{mathrsfs}
\usepackage{multirow}
\usepackage{setspace}
\usepackage{subfigure}
\usepackage{textcomp}
\usepackage[title]{appendix}
\usepackage{url}

\definecolor{myred}{RGB}{255,50,50}         
\definecolor{myblack}{RGB}{0,0,0}           
\definecolor{myblue}{RGB}{0,0,210}

\newcommand{\Rn}{\mathbb{R}^n}

\newcommand{\R}{\mathbb{R}}                      

\newtheorem{theorem}{Theorem}[section]

\newtheorem{definition}[theorem]{Definition}

\newtheorem{proposition}[theorem]{Proposition}

\newtheorem{assumption}[theorem]{Assumption}
\newtheorem{example}[theorem]{Example}

\numberwithin{equation}{section}

\usepackage{tikz}
\usetikzlibrary{arrows.meta, positioning}

\begin{document}


\title{A strong second-order sequential optimality condition for nonlinear programming problems}
\author{
  Huimin Li$^{*}$
  \and
  Yuya Yamakawa$^{*}$
  \and
  Ellen H. Fukuda$^{*}$
  \and
  Nobuo Yamashita%
  \thanks{Graduate School of Informatics, Kyoto University, 
  Yoshida-Honmachi, Sakyo-ku, Kyoto \mbox{606--8501}, Japan
    (\texttt{hm.li@amp.i.kyoto-u.ac.jp}, \texttt{yuya@i.kyoto-u.ac.jp}, \texttt{ellen@i.kyoto-u.ac.jp}, \texttt{nobuo@i.kyoto-u.ac.jp}).}
}

\maketitle


\begin{abstract}
    \noindent Most numerical methods developed for solving nonlinear programming problems are designed to find points that satisfy certain optimality conditions. While the Karush-Kuhn-Tucker conditions are well-known, they become invalid when constraint qualifications (CQ) are not met. Recent advances in sequential optimality conditions address this limitation in both first- and second-order cases, providing genuine optimality guarantees at local optima, even when CQs do not hold. However, some second-order sequential optimality conditions still require some restrictive conditions on constraints in the recent literature.
    In this paper, we propose a new strong second-order sequential optimality condition without CQs. We also show that a penalty-type method and an augmented Lagrangian method generate points satisfying these new optimality conditions. \\
    
\end{abstract}


\section{Introduction}
The nonlinear programming \eqref{NLP} problem plays a crucial role in the optimization theory. With the advancement of computational power, many researchers have developed numerical methods to address these NLPs. 
In particular, the primal-dual interior point method, the augmented Lagrangian method, and the sequential quadratic programming (SQP) method are well-known as numerical methods for NLPs, which are extensively studied and widely applied in NLPs.

Most optimization methods are designed to find points that satisfy the optimality conditions of the underlying problems. Optimality conditions are generally divided into two categories: sufficient conditions and necessary conditions. The set of points satisfying the sufficient condition is a subset of the optimal solution set, while the set of points satisfying the necessary condition always includes the optimal solution set. Although sufficient conditions provide the tightest set of points, they typically require strong assumptions, making them impractical for applications. 
In this paper, we focus on necessary conditions, which yield a larger set of points that includes the optimal solution set.

We consider two desirable properties that an ideal necessary condition should fulfill.
The first property is \textit{genuineness}, meaning that as a necessary condition, it does not require any additional assumptions, enhancing its applicability.
The second property is \textit{tightness}. 
A stronger necessary condition provides a tighter set of points, and the strongest necessary condition provides the tightest set of points, which is the exact optimal solution set. 
It is easy to see that an ideal necessary condition should satisfy both of these properties.

The well-known Karush-Kuhn-Tucker (KKT) conditions, for example, are well-established first-order necessary conditions, which hold under certain constraint qualifications (CQs) such as the Mangasarian-Fromovitz constraint qualification (MFCQ), Robinson's constraint qualification, and the constant rank constraint qualification (CRCQ), etc. Moreover, in convex optimization problems, the KKT conditions serve as also sufficient conditions for optimality.
If no CQ is satisfied, however, KKT conditions may fail to be necessary conditions. That is, there may exist local optima that do not satisfy the KKT conditions.

In the early 2010s, Andreani, Haeser and Mart\'inez~\cite{andreani2011sequential} introduced a new concept of necessary conditions for NLPs, known as Approximate KKT (AKKT) conditions, which do not rely on any CQs. The AKKT conditions, formulated using sequences, are sequential optimality conditions. These conditions are satisfied at local optima, regardless of whether any CQ holds, which makes AKKT genuine necessary conditions.
Furthermore, under certain CQs, they are equivalent to the KKT conditions, which shows that AKKT is at least not weak than KKT. The AKKT has also been extended to more general cases such as second-order cone programming (SOCP)~\cite{andreani2024optimality}, semidefinite programming (SDP)~\cite{andreani2020optimality}, mathematical programs with complementarity constraints (MPCC)~\cite{andreani2019new}, and mathematical programs with equilibrium constraints (MPEC)~\cite{ramos2021mathematical}.

More recently, second-order sequential optimality conditions have been proposed for NLPs, such as the second-order AKKT (AKKT2) ~\cite{andreani2017second}, which has also been extended to SOCP problems \cite{fukuda2025second},
and Strong-AKKT2 (SAKKT2) relative to a perturbed critical subspace $\tilde C$ ($\tilde C$-SAKKT2)~\cite{fischer2023achieving}. The AKKT2-type conditions require not only that the AKKT conditions are satisfied, but also incorporate the second-order information of the optimization problem, making them stronger than AKKT.
Furthermore, some algorithms that find points satisfying such sequential optimality conditions have also been proposed~\cite{andreani2012two,andreani2017second,andreani2010new,okabe2023revised,yamakawa2022stabilized}.
Although both AKKT2 and $\tilde C$-SAKKT2 have been established as necessary conditions for local minimizers in NLPs, 
we need to recall the key distinction between them. Broadly speaking, $\tilde C$-SAKKT2 is stronger than AKKT2 while as necessary conditions, $\tilde C$-SAKKT2 requires a relaxed constant rank assumption, which is weaker than the CRCQ, and AKKT2 does not. Figure \ref{fig: comparison} provides the comparison between AKKT2 and $\tilde C$-SAKKT2, where the notation $A \xlongrightarrow{C} B$ indicates that condition $B$ is a strengthening of condition $A$ with respect to property $C$. The same interpretation applies to Figure 2.

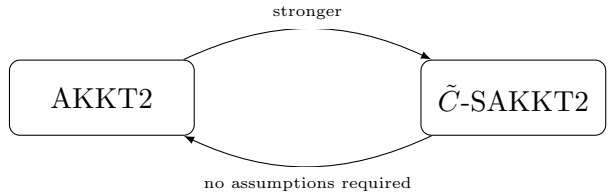
\begin{figure}[htbp]
    \centering
    \begin{tikzpicture}[node distance=3cm]
        \tikzset{
          line/.style={draw, -latex}, 
          box/.style={rectangle, rounded corners, minimum width=2.2cm, minimum height=1cm, text centered, draw=black, text width=2.2cm}
        }

        \node (akkt2) [box] {AKKT2};
        \node (sakkt2) [box, right=of akkt2] {$\tilde C$-SAKKT2};

        \draw[line, bend left=25] (akkt2) to node[fill=white, above] {\tiny stronger} (sakkt2);
        \draw[line, bend left=25] (sakkt2) to node[fill=white, below] {\tiny no assumptions required} (akkt2);
    \end{tikzpicture}
    \caption{Distinctions between AKKT2 and $\tilde C$-SAKKT2}
    \label{fig: comparison}
\end{figure}

Our study contributes by introducing a novel SAKKT2, which fulfills the two desirable properties of the ideal necessary condition. We show that the 
new proposed SAKKT2 is stronger than both AKKT2 and $\tilde C$-SAKKT2, and demonstrate that our proposed SAKKT2 holds at local minima without the need for any extra assumption, as compared to $\tilde C$-SAKKT2~\cite{fischer2023achieving}. The key here is to modify the standard regularized quadratic penalty problem typically used in such proofs. In addition, we propose algorithms that converge to second-order points and generate sequences whose limit points satisfy the SAKKT2 conditions. We show that both the penalty and the augmented Lagrangian methods generate SAKKT2 sequences under reasonable assumptions. In particular, for the penalty method, these assumptions can be replaced with more commonly used assumptions of trust region methods.

This paper is structured as follows. In Section 2, we review the existing optimality conditions for NLP, especially, we compare the existing AKKT2-type conditions, AKKT2 and $\tilde C$-SAKKT2. In Section 3, we propose the new SAKKT2, and show it strictly improves both AKKT2 and $\tilde C$-SAKKT2 conditions. In Section 4, we show that the SAKKT2 points could be generated by the penalty method and the augmented Lagrangian method under some assumptions. Finally, Section 5 concludes the paper.

\section{Existing optimality conditions} 

Before introducing the standard nonlinear programming \eqref{NLP} problems and discussing the related optimality conditions, let us first define the notations that will be used throughout this work. We consider the standard inner product in $\mathbb{R}^n$, given by $\langle a, b\rangle \coloneqq \sum_{i=1}^n a_i b_i$, and the Euclidean norm, given by $\|a\| \coloneqq \sqrt{\langle a, a\rangle}$, for every $a, b \in \mathbb{R}^n$. Let $\mathbb{R}_+$ denote the set of positive scalars. For a matrix $A\in \mathbb{R}^{n \times n}$, its transpose is denoted by $A^\top$. The identity matrix is denoted by $\mathbb{I}$, where its dimension is defined by the context. The gradient and the Hessian of a function $\psi: \mathbb{R}^n \rightarrow \mathbb{R}$ at an arbitrary point $x \in \mathbb{R}^n$ are represented by $\nabla \psi(x)$ and $\nabla^2 \psi(x)$, respectively. When $\tilde{\psi} \colon \mathbb{R}^n \times \mathbb{R}^{\ell} \rightarrow \mathbb{R}$, the gradient and the Hessian of $\tilde{\psi}$ at $(x, y) \in \mathbb{R}^n \times \mathbb{R}^{\ell}$ with respect to~$x$ are given by $\nabla_x \tilde{\psi}(x,y)$ and $\nabla^2_{xx} \tilde{\psi}(x,y)$, respectively.

In this paper, we deal with the following nonlinear programming problem:
\begin{equation}
    \begin{aligned}
        & \underset{x \in \mathbb{R}^n}{\text{minimize}}
        & & f(x) \\
        & \text{subject to}
        & & h_j(x) = 0 , \quad j = 1, \dots, p, \\
        & & & g_i(x) \leq 0, \quad i = 1, \dots, m,
      \end{aligned}
      \tag{NLP}
      \label{NLP}
\end{equation}
where the functions $f, h_j, g_i: \mathbb{R}^n \to \mathbb{R}$ are twice continuously differentiable on $\mathbb{R}^n$. We also use the notations $h := (h_1,\dots,h_p)$ and $g := (g_1,\dots,g_m)$. The feasible set of \eqref{NLP} is denoted by $\mathcal{F}$.

Moreover, for any $x \in \mathbb{R}^n$, we denote by $A(x)$ the set of indices of the inequality constraints that are active at $x$, i.e., 
$$
A(x):=\left\{i \in\{1, \ldots, m\} \mid g_i(x)=0\right\},
$$
and by $A(x)^C$ the complement of $A(x)$, i.e., $A(x)^C\coloneqq \{1, \dots, m\} \backslash A(x).$
The Lagrangian function $L: \Rn \times \mathbb{R}^p \times \mathbb{R}^m_+ \to \mathbb{R}$ for \eqref{NLP} is given by
$$
L(x, \mu, \omega)\coloneqq f(x)+h(x)^{\top} \mu+g(x)^{\top} \omega,
$$
where $\mu \in \mathbb{R}^p$ and $\omega \in \mathbb{R}^m_+$ are the Lagrange multipliers associated with the equality and inequality constraints, respectively. A well-known necessary optimality condition for~\eqref{NLP} is given below.
\begin{definition}
    A point $\bar x \in \mathbb{R}^n$ satisfies the Karush-Kuhn-Tucker (KKT) conditions for \eqref{NLP} if there exist $\bar \mu \in \mathbb{R}^p$ and $\bar \omega \in \mathbb{R}_+^m$ such that
\begin{equation*}
    \begin{gathered}
        \nabla_x L(\bar x, \bar \mu, \bar \omega) = \nabla f(\bar x) + \sum_{j=1}^p \bar \mu_j \nabla h_j(\bar x) + \sum_{i=1}^m \bar \omega_i \nabla g_i(\bar x) = 0,\\
        h_j (\bar x) = 0, \quad j = 1, \dots, p,\\
        \quad g_i(\bar x) \leq 0, \quad \bar \omega_i g_i(\bar x) = 0, \quad i = 1, \dots, m.
    \end{gathered}
\end{equation*}  
\end{definition}

If $(x, \mu, \omega)$ satisfies the KKT conditions, $x$ is called a KKT point. 
Although KKT conditions are often used as optimality conditions, they may not be satisfied at a local minimizer if no constraint qualification (CQ) holds at that point. One example of CQ is the constant rank constraint qualification (CRCQ) for \eqref{NLP}. To deal with the case where no CQ holds, the Approximate KKT (AKKT) conditions \cite{andreani2011sequential} have been proposed.

\begin{definition}
    A point $\bar{x}$ satisfies the approximate KKT (AKKT) conditions for \eqref{NLP} if there exists a sequence $\{(x^k, \mu^k, \omega^k, \varepsilon_k) \} \subseteq \mathbb{R}^n \times \mathbb{R}^p \times \mathbb{R}_{+}^m \times(0, \infty)$ with $x^k \rightarrow \bar{x}$ and $\varepsilon_k \searrow 0$ such that

    \begin{align}
\left\|\nabla f(x^k)+\sum_{j=1}^p \mu_j^k \nabla h_j(x^k)+\sum_{i=1}^m \omega_i^k \nabla g_i(x^k)\right\|  \leq \varepsilon_k, \label{AKKT_LagrangianVanish}\\
\|h(x^k)\|  \leq \varepsilon_k, \label{AKKT_FeasibilityOfEquality}\\
\|\max \{0, g(x^k)\}\| \leq \varepsilon_k, \label{AKKT_FeasibilityOfInequality}\\
\|\min \{\omega^k,-g(x^k)\}\|  \leq \varepsilon_k. \label{AKKT_Complementarity}
\end{align}
\end{definition}

Though the AKKT conditions require the first-order information of \eqref{NLP} as KKT conditions do, it should be noticed that AKKT conditions are sequential conditions, that is, they are defined by a sequence of points. The following proposition shows that, unlike KKT conditions, AKKT is a genuine necessary condition for \eqref{NLP} without requiring any assumptions.

\begin{proposition}
    {\cite[Theorem 2.1]{andreani2011sequential}}
If $\bar x$ is a local minimizer of \eqref{NLP}, then $\bar x$ satisfies the AKKT conditions.  
\end{proposition}

Unlike the KKT conditions which need a certain CQ, the AKKT conditions do not require any extra assumptions, which makes the AKKT a genuine necessary condition compared with the KKT conditions.

To exploit the second-order information of \eqref{NLP}, basic second-order necessary conditions have been proposed. To characterize the second-order optimality conditions for \eqref{NLP},
for any $x \in \mathcal{F}$, we need the critical cone 
$$C( x):=\left\{\begin{array}{l|l}
d \in \mathbb{R}^n & 
\begin{array}{l}
\nabla f( x)^\top d \leq 0 \\
\nabla h_j( x)^{\top} d=0, \quad\text{for } j=1, \ldots, p \\
\nabla g_i( x)^{\top} d \leq 0, \quad \text{for } i \in A( x) 
\end{array}
\end{array}\right\}.
$$
Additionally, the constant rank constraint qualification (CRCQ) introduced by~\cite{janin1984directional} is required, which is satisfied for $ x \in \mathcal{F}$ if there is a neighborhood of $ x$ such that for any subset $I \subseteq A( x)$, the rank of the family $\{\nabla h_j(y)\}^p_{j = 1} \cup \{\nabla g_i(y)\}_{i \in I}$ is constant for all $y$ in this neighborhood.

\begin{proposition}{\cite[Theorem 3.1]{andreani2010constant}}
    Let $\bar x$ be a local minimizer of \eqref{NLP} that satisfies CRCQ. Then there exist Lagrange multipliers $(\mu, \omega) \in \mathbb{R}^p \times \mathbb{R}_+^m$ such that
    $$d^\top \nabla^2_{xx} L(\bar x, \mu, \omega) d \geq 0 \quad \text{ for all } d \in C(\bar x).$$
\end{proposition}

The corresponding definition of the second-order sequential optimality conditions was firstly proposed in~\cite{andreani2017second}, and in~\cite{haeser2018some}, an equivalent characterization was given, which we include below. Note that in this case a perturbed critical subspace is used. 
\begin{definition}
A point $\bar{x}$ satisfies the second-order AKKT (AKKT2) conditions for \eqref{NLP} if there exists a sequence $\{(x^k, \mu^k, \omega^k, \varepsilon_k) \} \subseteq \mathbb{R}^n \times \mathbb{R}^p \times \mathbb{R}_{+}^m \times(0, \infty)$ with $x^k \rightarrow \bar{x}$ and $\varepsilon_k \searrow 0$ satisfying AKKT and 
      $$d^\top \nabla^2_{xx} L( x^k, \mu^k, \omega^k) d \geq -\varepsilon_k \|d\|^2\quad\text{ for all } d \in S(x^k, \bar x),$$
      where the perturbed critical space $S$ is defined as 
      $$ S(y, x):=\left\{\begin{array}{l|l}
d \in \mathbb{R}^n & 
\begin{array}{l}
\nabla h_j(y)^{\top} d=0, \quad \text{for } j=1, \ldots, p \\
\nabla g_i( y)^{\top} d = 0, \quad \text{for } i \in A( x) 
\end{array}
\end{array}\right\}.
$$
\end{definition}

The necessity of AKKT2 for a local minimizer of \eqref{NLP} has been proved in~\cite{andreani2017second}, which could also be found in~\cite{haeser2018some}. We include it here for later comparisons with other AKKT2-type conditions.

\begin{proposition}{\cite[Theorem 3.4]{andreani2007second}}
    If $\bar x$ is a local minimizer of \eqref{NLP}, then it satisfies AKKT2.
    \label{prop_necessityOfAKKT2}
\end{proposition}

It is easy to see that AKKT2 is stronger than AKKT since it does not only require that the AKKT holds, but also includes the second-order information of \eqref{NLP}. Furthermore, as a necessary condition, it improves the basic second-order necessary condition since it does not require any assumptions like CRCQ. We now present another AKKT2-type condition, the strong second-order AKKT with respect to the perturbed critical cone $\tilde C$ (referred to as $\tilde C$-SAKKT2), proposed in \cite{fischer2023achieving}.

\begin{definition}
     A point $\bar{x}$ satisfies the $\tilde C$-strong second-order AKKT ($\tilde C$-SAKKT2) conditions for \eqref{NLP} if there exists a sequence $\{(x^k, \mu^k, \omega^k, \varepsilon_k) \} \subseteq \mathbb{R}^n \times \mathbb{R}^p \times \mathbb{R}_{+}^m \times(0, \infty)$ with $x^k \rightarrow \bar{x}$ and $\varepsilon_k \searrow 0$ satisfying AKKT and 
      $$d^\top \nabla^2_{xx} L( x^k, \mu^k, \omega^k) d \geq -\varepsilon_k \|d\|^2 \quad \text{ for all } d \in \tilde C(x^k, \bar x, \omega^k),$$      
      where the perturbed critical cone $\tilde C$ is defined as 
$$\tilde C(y, x, \omega):=\left\{\begin{array}{l|l}
d \in \mathbb{R}^n & 
\begin{array}{l}
\nabla h_j(y)^{\top} d=0, \quad \text{ for } j=1, \ldots, p \\
\nabla g_i( y)^{\top} d \leq 0, \quad \text{ for } i \in A( x) \text{ with }  \omega_i = 0 \\
\nabla g_i( y)^{\top} d = 0, \quad \text{ for } i \in A( x) \text{ with }  \omega_i > 0
\end{array}
\end{array}\right\}.$$
\label{SAKKT2_NLP_strict}
\end{definition}

We also note that $ \tilde C(\bar x, \bar x, \bar \omega) = C (\bar x)$ when $(\bar x, \bar \mu, \bar \omega)$ is a KKT triple of \eqref{NLP}. Though $\tilde C$-SAKKT2 is obviously stronger than AKKT2, we include the proof here for completeness.

\begin{theorem}
    If a point $\bar x$ satisfies $\tilde C$-SAKKT2, then it satisfies $AKKT2$.
    \label{thm_CSAKKT2StrongerThanAKKT2}
\end{theorem}
\begin{proof}
    Since $\bar x$ is a $\tilde C$-SAKKT2 point, there exists a sequence $\{(x^k, \mu^k, \omega^k, \varepsilon_k)\} \subseteq \mathbb{R}^n \times \mathbb{R}^p \times \mathbb{R}_{+}^m \times(0, \infty)$ with $x^k \rightarrow \bar{x}$ and $\varepsilon_k \searrow 0$ satisfying AKKT conditions and 
      $$d^\top \nabla^2_{xx} L( x^k, \mu^k, \omega^k) d \geq -\varepsilon_k \|d\|^2 \quad \text{ for all } d \in \tilde C(x^k, \bar x, \omega^k).$$
    Given the relation $S(x^k, \bar x) \subseteq \tilde C(x^k, \bar x, \omega)$, the above second-order condition also holds for all $d \in S(x^k, \bar x)$, and hence $\bar x$ satisfies AKKT2.
\end{proof}

To establish the necessity of $\tilde C$-SAKKT2, the following relaxed constant rank condition is required.

\begin{assumption}
    It is said that a point $ x \in \Rn$ satisfies the relaxed constant rank condition if there is a neighborhood of $ x$ so that, for any subset $I \subseteq A( x)$, the rank of $\{\nabla g_i(y)\}_{i \in I}$ is constant for all $y$ in this neighborhood.
    \label{assumption_NLP}
\end{assumption}

It is important to note that Assumption \ref{assumption_NLP} only imposes restrictions only on the inequality constraint, making it a weaker condition than CRCQ. Based on this assumption, the necessity of $\tilde C$-SAKKT2 is given below.

\begin{proposition}{\cite[Theorem 3.2]{fischer2023achieving}}
    Let $\bar x$ be a local minimizer of \eqref{NLP} and suppose Assumption \ref{assumption_NLP} holds at $\bar x$. Then, $\bar x$ satisfies $\tilde C$-SAKKT2.
    \label{prop_necessityofC_SAKKT2}
\end{proposition}

Now we can compare the two existing AKKT2-type conditions. First, note that Theorem~\ref{thm_CSAKKT2StrongerThanAKKT2} implies that $\tilde C$-SAKKT2 is stronger than AKKT2. On the other hand, as necessary conditions, Propositions \ref{prop_necessityOfAKKT2} and \ref{prop_necessityofC_SAKKT2} indicate that AKKT2 does not require any assumptions while $\tilde C$-SAKKT2 does. 
In the next section, we will propose a new AKKT2-type condition, and demonstrate that it is stronger than both AKKT2 and $\tilde C$-SAKKT2, and prove that, as a necessary condition, it does not require any assumptions.


\section{Strong second-order sequential optimality conditions}

In this section, we present our results for \eqref{NLP}, which build upon and improve Proposition~\ref{prop_necessityofC_SAKKT2} from~\cite{fischer2023achieving}.
We first introduce the definition of the new strong AKKT2 condition.

\begin{definition}
         A point $\bar{x}$ satisfies the $\tilde S$-strong second-order AKKT ($\tilde S$-SAKKT2) conditions for \eqref{NLP} if there exists a sequence $\{(x^k, \mu^k, \omega^k, \varepsilon_k) \} \subseteq \mathbb{R}^n \times \mathbb{R}^p \times \mathbb{R}_{+}^m \times(0, \infty)$ with $x^k \rightarrow \bar{x}$ and $\varepsilon_k \searrow 0$, satisfying AKKT, and 
         \begin{equation}
             d^\top \nabla^2_{xx} L( x^k, \mu^k, \omega^k) d \geq -\varepsilon_k \|d\|^2\quad\text{ for all } d \in \tilde S(x^k, \bar x, \omega^k),
             \label{tildeSAKKT2_second}
         \end{equation}
         where the perturbed critical subspace $\tilde S$ is defined as
         $$\tilde S(y, x, \omega):=\left\{\begin{array}{l|l}
d \in \mathbb{R}^n & \begin{array}{l}
\nabla h_j(y)^{\top} d=0, \quad \text { for } j=1, \ldots, p, \\
\nabla g_i(y)^{\top} d = 0, \quad \text { for } i \in A(x) \text{ with } \omega_i > 0
\end{array} 
\end{array}\right\}.$$
\label{def_SAKKT2}
\end{definition}

The following result demonstrates that $\tilde S$-SAKKT2 is stronger than $\tilde C$-SAKKT2, and therefore also stronger than AKKT2.
\begin{theorem}
     If a point $\bar x$ satisfies $\tilde S$-SAKKT2, then it satisfies $\tilde C$-SAKKT2.
    \label{thm_S_SAKKT2StrongerThanC_SAKKT2}
\end{theorem}
\begin{proof}
    The relation $\tilde C (y, x, \omega) \subseteq \tilde S (y, x, \omega)$ completes the proof.
\end{proof}

It could be verified that $\tilde S$-SAKKT2 is strictly stronger than $\tilde C$-SAKKT2 via the following example.
\begin{example}
Consider the following nonlinear programming problem:
\begin{equation*}
    \begin{aligned}
        & \underset{x \in \mathbb{R}^2}{\mathrm{minimize}}
        & & f(x) = x_1 ^2 - x_2 ^2 \\
        & \mathrm{subject~to}
        & & g_1 (x) = - x_1 + x_2 \leq 0, \\
        & & & g_2(x) = - x_2 \leq 0.
      \end{aligned}
\end{equation*}
Let $\bar x = (0, 0)$ and define the sequences $x^k = (\frac{1}{k}, \frac{1}{k})$, $\omega^k = (0, 0)$, $\varepsilon_k = \frac{4}{k}$.
\end{example}
It is straightforward to verify that the AKKT conditions hold at $\bar x$ with the sequence $(x^k, \omega^k, \varepsilon_k)$.
We first show that $\tilde C$-SAKKT2 holds at $\bar x$. Note that the Hessian of the Lagrangian is 
$$\nabla^2_{xx} L( x^k, \omega^k) = \nabla^2 f(x^k) = 
\begin{bmatrix}
    2 & 0 \\
    0 & -2
\end{bmatrix},
$$
and that 
$$\nabla g_1(x^k) = 
\begin{bmatrix}
    -1 \\
    1
\end{bmatrix}, \quad
\nabla g_2(x^k) = 
\begin{bmatrix}
    0 \\
    -1
\end{bmatrix}.$$
Since both constraints are active at $\bar x$ and $\omega^k = (0, 0)$, the perturbed critical cone is given by 
$\tilde C(x^k, \bar x, \omega^k) = \{d \in \R^2 \mid d_1 \geq d_2 \geq 0 \}$.
One easily verifies that 
$$d^{\top} \nabla^2_{xx} L( x^k, \omega^k) d \geq 0 \geq -\varepsilon_k \|d\|^2 \text{ for all } d \in \tilde C(x^k, \bar x, \omega^k).$$
Thus, $\tilde C$-SAKKT2 holds at $\bar x$ with the chosen sequence. 
On the other hand, the perturbed critical subspace in $\tilde S$-SAKKT2 is $\tilde S (x^k, \bar x, \omega^k) = \R ^2$. By choosing $d = (0, 1)$, it is obvious that $\tilde S$-SAKKT2 fails at $\bar x$ with the chosen $( x^k, \omega^k, \varepsilon_k).$

We establish the necessity of $\tilde S$-SAKKT2 as follows.
\begin{theorem}
     If $\bar x$ is a local minimizer of \eqref{NLP}, then $\bar x$ satisfies $\tilde S$-SAKKT2.
     \label{thm_necessityOfS_SAKKT2}
\end{theorem}

\begin{proof}
Let $\delta \in (0, \frac{1}{3})$ be chosen such that $f(\bar{x}) \leq f(x)$ holds for all $x \in \mathcal{F} \cap \mathbb{B}(\bar x, \delta)$, where $\mathbb{B}(\bar x, \delta)$ denotes the closed ball with center $\bar x$ and radius $\delta$. Given a sequence $\left\{\rho_k\right\} \subset \mathbb{R}_{+}$ with $\rho_k \rightarrow+\infty$, we consider the following regularized penalty subproblem:
\begin{equation}
	\begin{aligned}
        & \underset{x \in \mathbb{R}^n}{\text{minimize}}
        & & F_k(x)\coloneqq f(x) + \frac{\rho_k}{2}  \sum_{j=1}^p
     h_j(x)^2 + \frac{\rho_k}{4}  \sum_{i=1}^m \operatorname{max}\{0, g_i(x)\}  ^4 + \frac{1}{4}\|x-\bar{x}\|^4  \\
        & \text{subject to}
        & &  x \in \mathbb{B}(\bar x, \delta).   
    \end{aligned}
    \label{sub}
\end{equation}

Let $x^k$ be a global solution of the optimization problem \eqref{sub}, which exists because its feasible set is nonempty and compact and the objective function is continuous. Therefore, for any $k \in \mathbb{N}$, we have
\begin{equation}
    f (x^k )+\frac{1}{4} \|x^k-\bar{x} \|^4 \leq  F_k (x^k ) \leq F_k(\bar{x})=f(\bar{x}).
    \label{leq}
\end{equation}
Moreover, because $\left\|x^k-\bar{x}\right\| \leq \delta$ is valid for all $k \in \mathbb{N}$, there exist $x^* \in \mathbb{B}(\bar x, \delta)$ and an infinite subset $K \subseteq \mathbb{N}$ such that $\lim _{K \ni k \to \infty } x^k=x^*$. 
\begin{sloppypar}
We will now show that $x^ * = \bar x$. By \eqref{leq}, we have 
\begin{equation}
    \frac{1}{2}  \sum_{j=1}^p
     h_j(x^k)^2 + \frac{1}{4}  \sum_{i=1}^m \operatorname{max}\{0, g_i(x^k)\}  ^4 \leq \frac{f(\bar x) - f(x^k) - \frac{1}{4}\|x^k-\bar{x}\|^4}{\rho_k}.
     \label{bounded}
\end{equation}
As $x^k \to x^*$ and $\rho_k \to \infty$, the right-hand side of \eqref{bounded} goes to zero, and hence 
~$\sum_{j=1}^p h_j(x^*)^2 = 0$ and $\sum_{i=1}^m \operatorname{max}\{0, g_i(x^*)\} ^4 = 0$. 
Now we note that $f(\bar{x}) \leq f(x^*)$ because $x^* \in \mathcal{F} \cap \mathbb{B}(\bar x, \delta)$. Then, by using \eqref{leq} and taking $k \to \infty$, we have 
$$f (x^* )+\frac{1}{4} \|x^*-\bar{x} \|^4 \leq  f(\bar{x}) \leq f(x^*).$$ Thus, we also obtain $x^* = \bar x$.
\end{sloppypar}
 
Therefore, as $x^k \to \bar x$, there is an $n_0 > 0$ such that for all $k \geq n_0$, $x^k \in \operatorname{int}(\mathbb{B}(\bar x, \delta))$,
and 
\begin{equation}\label{inactive_constrants}
    g_i(x^k) < 0, \quad i \in A(x^*)^C
\end{equation}
by the continuity of $g_i$, $i = 1, \dots, m$.
Note also that $\delta \in (0, \frac{1}{3})$ implies
\begin{equation}
    \|x^k - \bar x\|^3 \leq 3\|x^k - \bar x\|^2 \leq \|x^k - \bar x\|.
    \label{sufficiently_large_k}
\end{equation}
From now on, let $k$ be an arbitrary integer satisfying $k \geq n_0$. Using Fermat's rule, the gradient of the objective function of \eqref{sub} must vanish at $x^k$:
\begin{align}
    \nabla F_k (x^k) = &\, \nabla f(x^k) +  \sum_{j=1}^p \rho_k h_j(x^k) \nabla h_j(x^k) \nonumber \\
    &  +  \sum_{i=1}^m \rho_k  \operatorname{max}\{0, g_i(x^k)\}^3  \nabla g_i(x^k) 
    + \|x^k -\bar{x}\|^2 (x^k - \bar x) \label{first_order} \\
    = & \, 0. \nonumber
\end{align}   
Furthermore, we have
\begin{align}
    \nabla^2 F_k(x^k) & =\left( \nabla^2 f(x^k) +  \sum_{j=1}^p \rho_k h_j(x^k) \nabla^2 h_j(x^k) + 
    \sum_{i=1}^m \rho_k  \operatorname{max}\{0, g_i(x^k)\}^3 \nabla^2 g_i(x^k) \right) \nonumber \\    
    & + \sum_{j=1}^p \rho_k \nabla h_j(x^k) \nabla h_j(x^k) ^\top + 
    3\sum_{i=1}^m \rho_k  \operatorname{max}\{0, g_i(x^k)\}^2 \nabla g_i(x^k) \nabla g_i(x^k)^\top \label{second_order} \\
    & + 2 (x^k - \bar x) (x^k - \bar x)^\top 
    + \|x^k -\bar{x}\|^2 \mathbb{I} 
    \succeq  0. \nonumber
\end{align}
Let us define  
$$\mu^k_j \coloneqq \rho_k h_j(x^k), \quad \omega^k_i \coloneqq \rho_k  \operatorname{max}\{0, g_i(x^k)\}^3  ,$$
and 
$$\varepsilon_k \coloneqq \operatorname{max}\{\|x^k -\bar{x}\|, \|h(x^k)\|, \|\operatorname{max}\{0, g(x^k)\} \|\}.$$
Note that $\varepsilon_k \searrow 0$ ($k \to \infty$). By the definition of $\omega^k_i$, we have $\operatorname{min}\{\omega^k_i, -g_i(x^k)\} = \omega^k_i = 0$ for $g_i(x^k) \leq 0$ and $\operatorname{min}\{\omega^k_i, -g_i(x^k)\} = -g_i(x^k) < 0$ for $g_i(x^k) > 0$. Thus, 
$$\left\|\min \{\omega^k,-g (x^k )\}\right\| \leq \varepsilon_k.$$ Meanwhile, using \eqref{sufficiently_large_k} and \eqref{first_order}, we have 
$$
    \left\| \nabla f(x^k) +  \sum_{j=1}^p \mu^k_j \nabla h_j(x^k) + \sum_{i=1}^m \omega^k_i \nabla g_i(x^k) \right\| = \|x^k -\bar{x}\|^3 \leq \varepsilon_k,
$$
$$\|h(x^k)\| \leq \varepsilon_k.$$
Therefore, $\{(x^k, \mu^k, \omega^k, \varepsilon_k) \} $ satisfies AKKT. 

Let $d \in \tilde S (x^k, \bar{x}, \omega^k)$.
Notice that $\nabla h_j(x^k)^ \top d = 0$ for $j = 1, \dots, p.$
Inequality \eqref{second_order} and the definitions of $\mu^k$, $\omega^k$ and $\varepsilon_k$ give
\begin{equation}
    \begin{aligned}
    & d^\top  \left( \nabla^2 f(x^k) + \sum_{j=1}^p \mu^k_j \nabla^2 h_j(x^k)  + \sum_{i=1}^m \omega^k_i \nabla^2 g_i(x^k) \right) d  \\
    \geq &  - d^\top  \left(2 (x^k - \bar x) (x^k - \bar x)^\top + \|x^k -\bar{x}\|^2 \mathbb{I}\right) d  \\
    & - d^\top \left(3\sum_{i=1}^m \rho_k  \operatorname{max}\{0, g_i(x^k)\}^2  \nabla g_i(x^k) \nabla g_i(x^k)^\top\right)d.
    \end{aligned}   
    \label{inequality_perturbedremaining}
\end{equation}
Now, note that the last term in \eqref{inequality_perturbedremaining} can be written as
\begin{equation}
   \begin{aligned}
     & -3\sum_{i=1}^m \rho_k  \operatorname{max}\{0, g_i(x^k)\}^2  (\nabla g_i(x^k) ^\top d)^2  \\
     = & -3 \left( \sum_{\substack{i: i \in A(x^*)}}  \operatorname{max}\{0, g_i(x^k)\}^2  (\nabla g_i(x^k) ^\top d)^2  +  \sum_{\substack{i: i \in A(x^*)^C}}  \operatorname{max}\{0, g_i(x^k)\}^2  (\nabla g_i(x^k) ^\top d)^2 \right)  \\
     = & -3  \sum_{\substack{i: i \in A(x^*) \\ \text{ with } \omega^k_i > 0}}  \operatorname{max}\{0, g_i(x^k)\}^2  (\nabla g_i(x^k) ^\top d)^2   
   \end{aligned} 
   \label{inequality_cancled}
\end{equation}
since the definition of $\omega_i ^k$ gives the fact that $\omega_i ^k = 0$ is equivalent to $g_i(x^k) \leq 0$ and hence $\operatorname{max}\{0, g_i(x^k)\}^2 = 0$, and the last term in the second-line disappears from \eqref{inactive_constrants}.
Thus, we can then conclude that \eqref{inequality_cancled} is zero due to $d \in \tilde S (x^k, \bar{x}, \omega^k)$.
Then, we have
\begin{align*}
- d^\top  \nabla^2_x L(x^k, \mu^k, \omega^k) d \geq 
     &  - d^\top  \left(2 (x^k - \bar x) (x^k - \bar x)^\top + \|x^k -\bar{x}\|^2 \mathbb{I}\right) d \nonumber \\
  \geq & -\varepsilon_k \|d\|^2,
\end{align*}
    where the last line comes from Cauchy-Schwarz inequality and the fact that $3\|x^k - \bar x\|^2 \leq \varepsilon_k$ by \eqref{sufficiently_large_k}.
\end{proof}

Although $\tilde C$-SAKKT2 has been shown to be necessary conditions for a local minimizer under the relaxed constant rank condition in~\cite{fischer2023achieving}, we have demonstrated that this relaxed constant rank condition is actually not required, as shown by Theorems \ref{thm_S_SAKKT2StrongerThanC_SAKKT2} and \ref{thm_necessityOfS_SAKKT2}.
More specifically, Theorem \ref{thm_S_SAKKT2StrongerThanC_SAKKT2} establishes that $\tilde S$-SAKKT2 implies $\tilde C$-SAKKT2, and Theorem~\ref{thm_necessityOfS_SAKKT2} confirms that $\tilde S$-SAKKT2 itself is a necessary condition without requiring any assumptions. Thus, Assumption \ref{assumption_NLP} in Proposition \ref{prop_necessityofC_SAKKT2} can be removed without compromising its necessity.
The trick here is the use of quartic penalization in \eqref{sub} for handling inequality constraints. 

We conclude this section by comparing the three AKKT2-type conditions: AKKT2, $\tilde C$-SAKKT2, and our $\tilde S$-SAKKT2. As shown in Theorem \ref{thm_S_SAKKT2StrongerThanC_SAKKT2}, the $\tilde S$-SAKKT2 conditions we proposed in this paper is the strongest necessary condition among the three. Moreover, it does not require any assumptions as the necessary condition. These two desirable properties suggest that our new $\tilde S$-SAKKT2 should be regarded as the most ideal second-order sequential necessary condition of the three. Figure \ref{fig: contributions} shows the improvement of $\tilde S$-SAKKT2, compared with AKKT2 and $\tilde C$-SAKKT2.

\begin{figure}[htbp]
    \centering
    \begin{tikzpicture}[node distance=3.8cm]
        \tikzset{
          line/.style={draw, -latex}, 
          box/.style={rectangle, rounded corners, minimum width=2.2cm, minimum height=1cm, text centered, draw=black, text width=2.2cm}
        }

        \node (akkt2) [box] at (-3, 0) {AKKT2}; 
        \node (sakkt2) [box] at (3, 0) {$\tilde C$-SAKKT2}; 
        \node (result_nlp) [box] at (0, -4.5) {$\tilde S$-SAKKT2}; 

        \draw[line, bend left=25] (akkt2) to node[fill=white, above] {\tiny stronger} (sakkt2);
        \draw[line, bend left=25] (sakkt2) to node[fill=white, below] {\tiny no assumptions required} (akkt2);
        \draw[line] (akkt2) -- (result_nlp) node[pos=0.5, fill=white, align=center, text width=2.2cm] {\tiny stronger};
        \draw[line] (sakkt2) -- (result_nlp) node[pos=0.5, fill=white, align=center, text width=2.2cm] {\tiny stronger\\no assumptions};
    \end{tikzpicture}
    \caption{Comparison of $\tilde S$-SAKKT2 with AKKT2 and $\tilde C$-SAKKT2}
    \label{fig: contributions}
\end{figure}
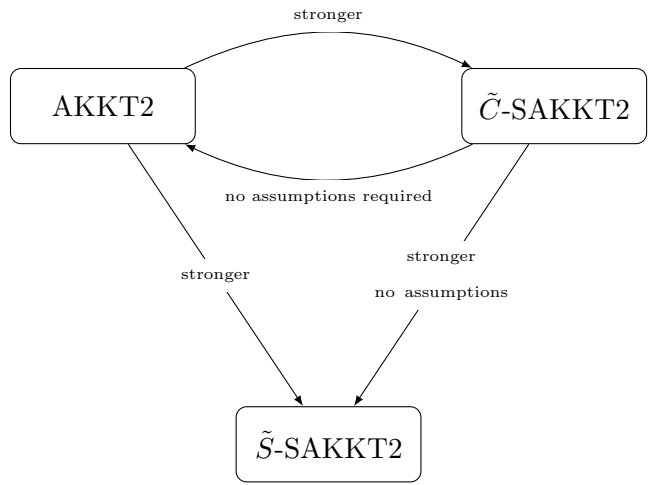

\section{Algorithms that generate $\tilde S$-SAKKT2 points}

In this section, we will present some algorithms that can generate $\tilde S$-SAKKT2 points. We begin by providing the following assumption, which will be needed in the proposed basic penalty method and the augmented Lagrangian method.

\begin{assumption}
A sequence $\{x^k \}$ satisfies the assumption if $$\lim_{k \to \infty} \sum_{j=1}^p h_j(x^k)^2 = 0 \quad \text{ and } \quad \lim_{k \to \infty}\sum_{i=1}^m \operatorname{max}\{0, g_i(x^k)\} = 0.$$
\label{assumption_limitfeasibility}
\end{assumption}

Note that the assumption is general, as it is required by many optimization algorithms. It is particularly satisfied by any algorithm if the accumulation point of the sequence generated by it is feasible.

\subsection{The penalty method}

To introduce the basic penalty method, let $\{\rho_k\}$ be a positive sequence, and consider the following unconstrained subproblem in each iteration~$k$:

\begin{equation}
	\begin{aligned}
        & \underset{x \in \mathbb{R}^n}{\text{minimize}}
        & & \phi_{\rho_k} (x)\coloneqq f(x) + \frac{\rho_k}{2}  \sum_{j=1}^p
 h_j(x)^2 + \frac{\rho_k}{4}  \sum_{i=1}^m \operatorname{max}\{0, g_i(x)\}  ^4.
   \end{aligned}
   \label{penalty_subproblem}
\end{equation}
Observe that $\phi_{\rho_k}$ is similar to the well-known quadratic penalty function, except that the violation of the inequality constraint is raised to the fourth power, rather than squared. A basic penalty method that uses this function and generates $\tilde S$-SAKKT2 points under Assumption \ref{assumption_limitfeasibility} is given in Algorithm~\ref{algm_penalty}.

\begin{algorithm}
\caption{Basic penalty method}
\begin{algorithmic}[1]
    \STATE Choose $\{\varepsilon_k\}, \{\rho_k\} \in (0, \infty)$ such that $\varepsilon_k \searrow 0$ and $\rho_k \to \infty$, and $x^0 \in \Rn$.
    Set $k := 0$.
    \STATE Find an approximate solution $x^k$ of \eqref{assumption_limitfeasibility} satisfying 
    \begin{equation}
        \|\nabla \phi_{\rho_k} (x^k) \| \leq \varepsilon_k \quad \text{ and } \quad d^\top \nabla^2 \phi_{\rho_k}(x^k) d \geq -\varepsilon_k \|d\| ^2 \text{ for all } d \in \Rn.
        \label{penalty_condition}
    \end{equation}
    \STATE Set $k \coloneqq k + 1$ and go back to Step~2.
\end{algorithmic}
\label{algm_penalty}
\end{algorithm}

Note that Step~2 of Algorithm~\ref{algm_penalty} can be performed using a trust region method. The next theorem shows that the sequence generated by Algorithm 1 satisfies $\tilde S$-SAKKT2 under Assumption 4.1.

\begin{theorem}
  \label{theo:basic_penalty}
  Let $\{x^k \}$ be a sequence generated by Algorithm \ref{algm_penalty} and $x^*$ be any of its accumulation points. Suppose $\{x^k \}$ satisfies Assumption \ref{assumption_limitfeasibility}. Then, $x^*$ is a $\tilde S$-SAKKT2 point.
\end{theorem}
\begin{proof}
    Condition \eqref{penalty_condition} gives 
    \begin{equation}
        \left\| \nabla f(x^k) +  \sum_{j=1}^p \rho_k h_j(x^k) \nabla h_j(x^k) 
  +  \sum_{i=1}^m \rho_k  \operatorname{max}\{0, g_i(x^k)\}^3  \nabla g_i(x^k) \right\| \leq \varepsilon_k
    \label{penalty_LagrangeVanish}
    \end{equation}
and
\begin{align}
    & d ^\top \left( \nabla^2 f(x^k) +  \sum_{j=1}^p \rho_k h_j(x^k) \nabla^2 h_j(x^k) + 
    \sum_{i=1}^m \rho_k  \operatorname{max}\{0, g_i(x^k)\}^3 \nabla^2 g_i(x^k) \right) d \nonumber \\    
    & + \sum_{j=1}^p \rho_k \| \nabla h_j(x^k) ^\top d \|^2 + 
    3\sum_{i=1}^m \rho_k  \operatorname{max}\{0, g_i(x^k)\}^2 \| \nabla g_i(x^k)^\top d\|^2 \geq -\varepsilon_k \|d\|^2
    \label{penalty_condition_second_order}
\end{align}
for all $d \in \Rn$. Then by setting the Lagrange multipliers associated to $x^k$ as 
    $$\mu^k_j \coloneqq \rho_k h_j(x^k), \quad \omega^k_i \coloneqq \rho_k  \operatorname{max}\{0, g_i(x^k)\}^3,$$
and by \eqref{penalty_LagrangeVanish} as well as Assumption \ref{assumption_limitfeasibility}, we have 
\eqref{AKKT_LagrangianVanish}, \eqref{AKKT_FeasibilityOfEquality}, and \eqref{AKKT_FeasibilityOfInequality} are satisfied. By the definition of $\omega^k$, when $g_i(x^k) \leq 0$, we have $\omega_i^k = 0$ and hence $\min \{\omega_i^k, -g_i(x^k)\} = \omega_i^k = 0$. When $g_i(x^k) \geq 0$, on the other hand, we have
\[
\lim_{k\to \infty} \min \{\omega_i^k, -g_i(x^k)\} = -\lim_{k\to \infty} g_i(x^k) = 0
\]
due to Assumption \ref{assumption_limitfeasibility}, so that \eqref{AKKT_Complementarity} is satisfied. Also, similarly to the proof of Theorem~\ref{thm_necessityOfS_SAKKT2}, we can see that by choosing any $d \in \tilde{S}(x^k, \bar{x}, \omega^k) \subset \mathbb{R}^n$, \eqref{penalty_condition_second_order} implies \eqref{tildeSAKKT2_second}.
Since $x^*$ is an accumulation point, there is an infinite subset $K \subseteq \mathbb{N}$ such that $\lim_{K \ni k \to \infty} x^k = x^*$. Then, we have $\{(x^k, \mu^k, \omega^k)\} \subseteq \mathbb{R}^n \times \mathbb{R}^p \times \mathbb{R}_{+}^m$ with $x^k \to x^*$ satisfying $\tilde{S}$-SAKKT2 conditions and hence $x^*$ is a $\tilde{S}$-SAKKT2 point.
\end{proof}

Next we present a modification of Algorithm 1, which does not require  Assumption 4.1. The modified penalty method applies the trust region method from an appropriate initial point. And this approach demonstrates that Assumption \ref{assumption_limitfeasibility} automatically holds for the penalty method given in Algorithm~\ref{penalty_algm2}.

\begin{algorithm}
\caption{Modified penalty method}
\begin{algorithmic}[1]
    \STATE Choose $\{\varepsilon_k\}, \{\rho_k\} \in (0, \infty)$ such that $\varepsilon_k \searrow 0$ and $\rho_k \to \infty$, and a feasible point $x^0 \in \mathcal{F}$.
    Set $\hat x := x^0$ and $k := 0$.
    \STATE Obtain an approximate solution $x^k$ of \eqref{assumption_limitfeasibility} satisfying 
    $$\|\nabla \phi_{\rho_k} (x^k) \| \leq \varepsilon_k \quad \mbox{and} \quad d^\top \nabla^2 \phi_{\rho_k}(x^k) d \geq -\varepsilon_k \|d\|^2 \mbox{ for all } d \in \Rn$$ via a trust region method starting from $\hat x$.
    \IF{$f(x^k) + \frac{\rho_{k+1}}{2}  \sum_{j=1}^p h_j(x^k)^2 + \frac{\rho_{k+1}}{4}  \sum_{i=1}^m \operatorname{max}\{0, g_i(x^k)\}^4 \leq f(x^0)$}
        \STATE set $\hat x := x^k$.
    \ELSE
        \STATE set $\hat x := x^0$.
    \ENDIF
    \STATE Set $k \coloneqq k + 1$ and go back to Step~2.
\end{algorithmic}
\label{penalty_algm2}
\end{algorithm}

The following result demonstrates that the modified penalty method can generate $\tilde S$-SAKKT2 points without the need for Assumption~\ref{assumption_limitfeasibility}.

\begin{theorem}
  Let $\{x^k \}$ be a sequence generated by Algorithm \ref{penalty_algm2} and $x^*$ be any of its accumulation points. Then, we have 
  $$\lim_{k \to \infty} \sum_{j=1}^p h_j(x^k)^2 = 0 \quad \mbox{and} \quad \lim_{k \to \infty}\sum_{i=1}^m \operatorname{max}\{0, g_i(x^k)\} = 0,$$
  i.e., Assumption~\ref{assumption_limitfeasibility} automatically holds. 
  Furthermore, $x^*$ is a $\tilde S$-SAKKT2 point.
\end{theorem}

\begin{proof}
Since $x^*$ is an accumulation point of~$\{x^k \}$, there is an infinite subset $K \subseteq \mathbb{N}$ such that $\lim_{K \ni k \to \infty} x^k = x^*$. 
Meanwhile, since the trust region method is a descent method~(see \cite[Chapter 4]{wright2006numerical}), at the iteration $k$ we have
$$\phi_{\rho_k} (x^k) \leq \phi_{\rho_{k}} (\hat x).$$
 Also, from the algorithm, either $\hat{x} = x^0$ or $\hat{x} = x^{k-1}$.
 In the former case, $\phi_{\rho_{k}} (\hat x) = \phi_{\rho_{k}} (x^0) = f(x^0)$ because $x^0 \in \mathcal{F}$. In the latter case, $\phi_{\rho_{k}} (\hat x) = \phi_{\rho_{k}} (x^{k-1}) \le f(x^0)$ from the previous iteration and Step~3. 
 Thus, we obtain
 $$f(x^k) + \frac{\rho_k}{2}  \sum_{j=1}^p
 h_j(x^k)^2 + \frac{\rho_k}{4} \sum_{i=1}^m \operatorname{max}\{0, g_i(x^k)\}  ^4  \leq \phi_{\rho_{k}} (\hat x) \leq f(x^0).$$
 Therefore,
 \begin{equation*}
     \frac{1}{2}  \sum_{j=1}^p
 h_j(x^k)^2 + \frac{1}{4}  \sum_{i=1}^m \operatorname{max}\{0, g_i(x^k)\}  ^4  \leq \frac{f(x^0) - f(x^k)}{\rho_k}.
 \end{equation*}
 Taking the limit $K \ni k \to \infty$, the right-hand side of the above inequality goes to $0$ because $\rho_k \to \infty$ and $f$ is continuous. Hence Assumption~\ref{assumption_limitfeasibility} holds. The remainder of the proof is similar to that of  Theorem~\ref{theo:basic_penalty}.
\end{proof}

Algorithm \ref{penalty_algm2} demonstrates that when using the modified penalty method, Assumption~\ref{assumption_limitfeasibility} can be replaced by other general assumptions commonly applied in the trust region method (see, for example, \cite[Theorem 4.1]{wright2006numerical}), without affecting the global convergence.

\subsection{The augmented Lagrangian method}
Consider the following augmented Lagrangian function:
\begin{equation}
    L_\rho(x, \mu, \omega )\coloneqq f(x) + \frac{\rho}{2}  \sum_{j=1}^p
 \left[h_j(x) + \frac{\mu_j}{\rho} \right]^2 + \frac{\rho}{4}  \sum_{i=1}^m \max\left\{0, g_i(x) + \frac{\omega_i}{\rho}\right\}  ^4 
\end{equation}
for all $x \in \Rn$, $\rho > 0$ and where $\mu \in \mathbb{R}^p$ and $\omega \in \mathbb{R}^m_+$ are the Lagrange multipliers. As in the case of the penalty method, the function above is similar to the classical Powell-Hestenes-Rockafellar augmented Lagrangian function \cite{hestenes1969multiplier, powell1969method, rockafellar1974augmented}, except that the violations of the inequality constraints are raised to the fourth power.

\begin{algorithm}
\caption{Augmented Lagrangian method}
\begin{algorithmic}[1]
    \STATE Let $\mu_{\min }<\mu_{\max }$, $\omega_{\max }>0$, $\gamma>1$, $\rho_1>0$, $\tau \in(0,1)$. Let $\{\varepsilon_k\} \in (0, \infty)$ such that $\varepsilon_k \searrow 0$.
    Let $\mu_j^1 \in [\mu_{\min }, \mu_{\max }]$ for all $ j \in\{1, \ldots, p\}$ and $\omega_i^1 \in[0, \omega_{\max }]$ for all $i \in\{1, \ldots, m\}$. Let $x^0 \in \mathbb{R}^n$, and $V^0\coloneqq \max \{0, g(x^0)\}$. Set $k := 0$.
    
    \STATE Obtain an approximate minimizer $x^k$ of $L_{\rho_k}(x, \mu^k, \omega^k)$ satisfying
    \begin{equation}
        \|\nabla_x L_{\rho_k}(x^k, \mu^k, \omega^k)\| \leq \varepsilon_k \quad \text { and } \quad \nabla^2_{xx} L_{\rho_k}(x^k, \mu^k, \omega^k) \geq  -\varepsilon_k \mathbb{I}.
        \label{augmented_condition}
    \end{equation}
    \STATE Define $V_i^k \coloneqq \max \{g_i(x^k),-\omega_i^k / \rho_k\}$ for $i \in\{1, \ldots, m\}$.
    \IF{$\max \{\|h(x^k)\|_{\infty},\|V^k\|_{\infty}\} \leq \tau \max \{\|h(x^{k-1})\|_{\infty},\|V^{k-1}\|_{\infty}\}$}
        \STATE set $\rho_{k+1}:=\rho_k$,
    \ELSE \STATE set $\rho_{k+1}:=\gamma \rho_k$.
    \ENDIF
    \STATE Compute $\mu_j^{k+1}\coloneqq \max(\mu_{\min}, \min(\mu_{\max}, \mu_j^k + \rho_k h_j(x^k))) \in [\mu_{\min }, \mu_{\max }]$ for all $j \in\{1, \ldots, p\}$ and $\omega_i^{k+1}\coloneqq \max(0, \min(\omega_{\max}, \omega_i^k + \rho_k g_i(x^k))) \in[0, \omega_{\max }]$, for all $i \in\{1, \ldots, m\}$. Set $k \coloneqq k + 1$ and go back to Step 1. 
\end{algorithmic}
\label{augmented_algm}
\end{algorithm}

The augmented Lagrangian method that we consider is given in Algorithm~\ref{augmented_algm}, which is based on ALGENCAN (see \cite[Algorithm 4.1]{andreani2007second}). The following result shows that Algorithm~\ref{augmented_algm} generates $\tilde S$-SAKKT2 points under Assumption \ref{assumption_limitfeasibility}.

\begin{theorem}
    Let $\{x^k \}$ be a sequence generated by Algorithm \ref{augmented_algm} and $x^*$ be any of its accumulation points. Suppose that $\{x^k \}$ satisfies Assumption \ref{assumption_limitfeasibility}. Then, $x^*$ is a $\tilde S$-SAKKT2 point.
\end{theorem}

\begin{proof}
    Since $x^*$ is an accumulation point, there is an infinite subset $K \subseteq \mathbb{N}$ such that for all $k \in K$, we have $\lim_{K \ni k\to \infty} x^k = x^*$. 
     To see that $x ^*$ fulfills the $\tilde S$-SAKKT2 conditions, first note that \eqref{augmented_condition} gives 
\begin{equation}
    \| \nabla_x L_{\rho_k}(x^k, \hat \mu^k, \hat \omega^k) \| \leq \varepsilon_k
\end{equation}
and
\begin{align}
  \nabla^2 L_{\rho_k}(x^k, \hat \mu^k, \hat \omega^k)  
     & + \sum_{j=1}^p \rho_k \nabla h_j(x^k) \nabla h_j(x^k) ^\top \nonumber \\ & + 
    3\sum_{i=1}^m \rho_k  \operatorname{max}\{0, g_i(x^k) + \frac{\omega^k_i}{\rho^k}\}^2 \nabla g_i(x^k) \nabla g_i(x^k)^\top \succeq -\varepsilon_k \mathbb{I}
    \label{augmented_soc}  
\end{align}
where $\hat \mu^k_j \coloneqq \mu^k_j + \rho_k h_j(x^k)$ and $\hat \omega^k_i \coloneqq \rho_k \max \left\{0, g_i(x^k) + \frac{\omega^k_i}{\rho_k}\right\}^3$.

It has been shown in \cite[Theorem 4.1]{andreani2007second} that for $k$ large enough, $\hat \omega^k_i = 0$ for all $i \notin A(x^*)$. Moreover, by the definition of $\hat \omega^k_i$, we notice that $\hat \omega^k_i = 0$ if and only if $g_i(x^k) + \frac{\omega^k_i}{\rho_k} \leq 0$, which means the \eqref{augmented_soc} could be rewritten as 
\begin{align}
  & \nabla_x^2 L_{\rho_k}(x^k, \hat \mu^k, \hat \omega^k)  
      + \sum_{j=1}^p \rho_k \nabla h_j(x^k) \nabla h_j(x^k) ^\top \nonumber \\ & + 
    3\sum_{i: i \in A(x^*) \text{ with } \hat \omega^k_i > 0}^m \rho_k  \operatorname{max}\left\{0, g_i(x^k) + \frac{\omega^k_i}{\rho_k}\right\}^2 \nabla g_i(x^k) \nabla g_i(x^k)^\top \succeq -\varepsilon_k \mathbb{I}.
    \label{augmented_soc_trans}  
\end{align}
Then by choosing any $d \in \tilde S(x^k, x^*, \hat \omega^k)$, we get 
$$d ^\top \nabla_x^2 L_{\rho_k}(x^k, \hat \mu^k, \hat \omega^k) d \geq \varepsilon_k \|d\|^2,$$ which shows that the augmented Lagrangian method generates $\tilde S$-SAKKT2 sequences.
\end{proof}


\section{Conclusion}

In this paper, we introduced a new strong second-order sequential optimality condition for \eqref{NLP}. This condition strengthens existing sequential optimality conditions by offering a more tight framework that does not rely on restrictive assumptions. Our work demonstrates that this condition holds as a necessary condition for local minima without the need for constraint qualifications, further broadening its applicability in various \eqref{NLP} settings.

Additionally, we established the global convergence of penalty methods and an augmented Lagrangian method under weak assumptions, as it shows that these methods can be effectively applied to solve \eqref{NLP} problems while generating points that satisfy the proposed strong second-order sequential optimality condition. The introduction of quartic penalization for inequality constraints further enhances the practicality of these algorithms. 

For future work, we can explore other optimization methods that generate $\tilde S$-SAKKT2 points. In particular, we plan to investigate the use of the Sequential Quadratic Programming (SQP) method that may generate such points. This extension would further enhance the flexibility and applicability of our framework, allowing for more diverse algorithms to be used in solving \eqref{NLP} problems. Another idea is to define analogous strong 
second-order conditions using the so-called complementarity AKKT (CAKKT)~\cite{haeser2018second}.

\section*{Acknowledgements} 
This work was supported by the JST SPRING (JPMJSP2110), Japan Society for the Promotion of Science, 
Grant-in-Aid for Scientific Research (C) (JP19K11840, JP21K11769) and 
Grant-in-Aid for Early-Career Scientists (JP21K17709).


\bibliographystyle{plain}
\bibliography{references.bib}

\end{document}